\documentclass[11pt,english]{amsart}
\usepackage[T1]{fontenc}
\usepackage[utf8]{luainputenc}
\usepackage{xcolor}
\usepackage{pdfcolmk}
\usepackage{amsthm}
\usepackage{amsmath}
\usepackage{stmaryrd}
\usepackage{stackrel}
\usepackage{setspace}
\PassOptionsToPackage{normalem}{ulem}
\usepackage{ulem}

\makeatletter

\providecolor{lyxadded}{rgb}{0,0,1}
\providecolor{lyxdeleted}{rgb}{1,0,0}

\numberwithin{figure}{section}
\numberwithin{equation}{section}
\numberwithin{table}{section}

  \input amssymb.sty
\usepackage{etoolbox}
\patchcmd{\thebibliography}{\section*}{\section}{}{}

\usepackage{algpseudocode} \usepackage{algorithm}
\usepackage{algpseudocode,algorithm}
\usepackage{lipsum}
\usepackage{caption}
\usepackage{graphics}

\usepackage{amsmath}
\usepackage{amssymb}
\usepackage[all]{xy}

\usepackage{epsfig}\usepackage{youngtab}

\newcommand{\ef}{\end{equation}}
\chardef\bslash=`\\ 

\newdir{:=}{{}}
\newcommand*\colvec[3][]{
    \begin{pmatrix}\ifx\relax#1\relax\else#1\\\fi#2\\#3\end{pmatrix}
}
\newtheorem{thm}{Theorem}[section]
\newtheorem*{thm*}{Theorem}

\newtheorem{lem}{Lemma}[section]
\newtheorem*{lem*}{Lemma}
\newtheorem{corl}{Corollary}[lem]
\newtheorem*{corl*}{Corollary}

\newtheorem{prop*}{Proposition}

\theoremstyle{definition}
\newtheorem{defn}{Definition}[section]
\newtheorem{examp}{Example}
\newtheorem*{examp*}{Example}
\newtheorem*{remark*}{Remark}
\newtheorem*{CC*}{Crossover Conjecture}
\newtheorem*{Note*}{Note}
\newtheorem*{defn*}{Definition}

 \theoremstyle{remark}
\newtheorem{remark}{Remark}[section]

 \renewcommand{\sectionmark}[1]{}

\newcommand{\la}{\langle}
\newcommand{\ra}{\rangle}

\newcommand{\defect}{\operatorname{def}}

\renewcommand{\a}{\alpha}

\makeatother

\usepackage{babel}
\begin{document}

\title[EXTERNAL LITTELMANN PATHS OF TYPE $A$]{EXTERNAL LITTELMANN PATHS FOR CRYSTALS  OF TYPE $A$}
\author{Ola Amara-Omari$^1$ and  Mary Schaps}
\address { Bar-Ilan University, Ramat-Gan, Israel,
mschaps@macs.biu.ac.il}

\address{ Bar-Ilan University, Ramat-Gan, Israel,
 olaomari77@hotmail.com}
 \subjclass[2010] {17B10, 17B35}
\thanks {$^1$Partially supported by Ministry of Science, Technology and Space fellowship , at Bar-Ilan University.}

\maketitle

 \begin{abstract} 
For the Kashiwara crystal of a highest weight representation of an affine Lie algebra of  type $A$ and rank $e$, with highest weight $\Lambda$, there is a labeling by multipartitions and by piecewise linear paths in the real weight space called Littelmann paths.  Both labelings are constructed recursively, but since Kashiwara demonstrated that the crystals are isomorphic, there is a bijection between the labels.  

We choose a multicharge $(k_1,\dots,k_r)$, with $0 \leq k_1\leq  k_2....\leq k_r \leq e-1$. We put $k_i$ in the node at the upper left corner of partition $i$ of the multipartition and let the residues from $\mathbb Z/ e \mathbb Z$ increase across rows and decrease down columns.    For $e=2$, we call a multipartition residue-homogeneous if all nonzero rows end in nodes of the same residue and partitions with the same corner residue have first rows of the same parity. It is strongly residue homogeneous if each partition ends in a triangle of whose side has length one less than the first row of the next partition.

In this paper we show that each such multipartition corresponds to a Littelmann path which is unidirectional in the sense that the projection of the the main part of the path to the coordinates of the fundamental weights consists of long paths all lying in either the second or fourth quadrant, separated by oscillating paths with a fixed integer oscillator. The path corresponding to such a multipartition can be constructed non-recursively using only integers describing the structure of the multipartition.
\end{abstract}

\section{INTRODUCTION}

Let $e$ be the rank of an affine Lie algebra  $\mathfrak g$ of type A and let $r$ be the level of a highest weight module with highest weight $\Lambda$. For type A, the level $r$  is the sum of the coefficients of the fundamental weights in the dominant integral weight  $\Lambda$. There is a labelling of the elements of the basis of the highest weight module by multipartitions and those multipartitions which are recursively constructed as labels of the basis elements are called $e$-regular.  The level $r$ is the common number of partitions in these multipartitions.

 The problem of determining the $e$-regular multipartions of type $A$ in a non-recursive manner has shown only slow progress. It is known for $r=1$.  Mathas settled it for $e=2$ in \cite{M}, and Ariki, Kreiman and Tsuchioka for $r=2$ in \cite{AKT}.  A few other results extending these are also available.  We are trying to attack the problem through Littelmann paths, by trying to find a direct way to translate between the multipartition labelling  a crystal basis element and the corresponding Littelmann path. The results in this paper are for the $e=2$ case, for which, as we mentioned, Mathas already solved the problem of finding a criterion for $e$-regular multipartitions, but we hope that finding a way to translate from Littelmann paths to multipartitions will be relevant to the problem of $e$-regularity, since only $e$-regular multipartitions will correspond to a Littelmann path.

In \S2 and \S3, we give needed background, and the definition of a unidirectional Littelmann path.  In \S 4, we define a strongly residue homogeneous multipartition and prove our main result, that a strongly residue-homogeneous multipartition corresponds to a unidirectional Littelman path whose invariants can be determined from the structure of the multipartition. In \S5, we give examples to illustrate the limitations of the theorem and explain what remains open.

\section{DEFINITIONS AND NOTATION}

Let $\mathfrak g$ be the affine Lie algebra $A^{(1)}_{e-1}$. Let $C$ be the Cartan matrix, and $\delta$ the null root. Let $\Lambda$ be a dominant integral weight, let $V(\Lambda)$ be the highest weight module with that 
highest weight, and let $P(\Lambda)$ be the set of weights of $V(\Lambda)$. Let $I$ be the set of residues $\mathbb Z / e \mathbb Z$ and let  $Q$ be the $\mathbb Z$-lattice generated by the simple roots,
\[
\a_0,\dots,\a_{e-1}.
\]
Let  $Q_+$ be the subset of $Q$ in which all coefficients are non-negative. 

The weight space $P$ of the affine Lie algebra has two different bases.  One is given by the fundamental weights together with the null root, $\Lambda_0,\dots, \Lambda_{e-1}, \delta$, 
and one is given by $\Lambda_0, \a_0,\dots,\a_{e-1}$.  We will usually use the first basis for our weights.  

A highest weight module $V(\Lambda)$ is integrable.  Every weight $\lambda$ has the form $\Lambda-\alpha$, for $\alpha \in Q_+$.  The vector
of nonnegative integers giving the coefficients of $\alpha$ is called the content of $\lambda$. We will not repeat all the standard material about the symmetric form $(- \mid -)$, which can be found in \cite{Ka}. We follow \cite{Kl} in defining the defect of a weight by

\[
\defect(\lambda)=\frac{1}{2}((\Lambda \mid \Lambda)-  (\lambda \mid \lambda)).
\]

 Since we are in a highest weight module, we always have $(\Lambda \mid \Lambda) \geq (\lambda \mid \lambda)$, and the defect is in fact an integer for the affine Lie algebras of type $A$ treated in this paper. The weights of defect $0$ are those lying in the Weyl group orbit of $\Lambda$ and will play an important role in the definition of the Littelmann paths. 
Define
\[
\max P(\Lambda)=\{\lambda \in P(\Lambda) \mid \lambda + \delta \not\in P(\Lambda)\},
\] 
and by \cite{Ka}, every element of 
$P(\Lambda)$ is of the form $\{y-k\delta \mid y \in \max P(\Lambda), k \in \mathbb Z_{ \geq 0}\}$.
Let $W$ denote the Weyl group, generated by reflections $s_0,\dots, s_{e-1}$.

By the ground-breaking work of Chuang and Rouquier \cite{CR}, the highest weight module $V(\Lambda)$ can be categorified.  The weight spaces lift to categories of representations of  blocks of cyclotomic Hecke algebras, the basis vectors lift to simple modules,  the Chevalley generators $e_i,f_i$ lift to restriction and induction functors $E_i,F_i$, and the simple reflections in the Weyl group lift to derived equivalences, which in a few important cases are actually Morita equivalences.  We will not be using the categorified version, but it provides the underlying motivation for trying to understand the multipartitions, which label the simple modules of the cyclotomic Hecke algebras. For $r=1$, in what is called the degenerate case, these are the group algebras of symmetric groups.

There are three distinct labellings for an element of a Kashiwara crystal $B(\Lambda)$, \cite{K1},\cite{K2}, of type $A$.  We will not directly need the definition of the crystal or of its crystal base. An element of the crystal base can be labelled  by its Littelmann path,  by its multipartitions, and by its canonical basis.  In this paper we are concerned only with the first two, looking for cases where there is a direct connection between them, so that we can read off the multipartition from the Littelmann path or the Littelmann path from the  multipartition.  In a different paper, \cite{OS}, we make a similar direct passage from a multipartition to a canonical basis element for a symmetric crystal.

The theory originated with work of Lakshmibai and Seshadri for type A, but it was Littelmann in \cite{L} who extended to other types and  who proved many of the most important properties of the path model of the highest weight representations. 
A path $\pi$ is a continuous, piecewise linear function from the closed real interval $[0,1]$ into the real space $\mathbb R \otimes_\mathbb Z P(\Lambda)$ such that $\pi(0)=\{0\}$ and $\pi(1) \in P(\Lambda)$.  The weight $\pi(1)$   will be called the \textit{weight of the path}.

 For any residue $\epsilon$ in the set of residues $I$, we define a function $H_\epsilon^\pi(t)=\la \pi(t),h_\epsilon \ra$, which is simply the projection of the path onto the coefficient of $ \Lambda_\epsilon$.  We then set
\[m_\epsilon=\min_t (H^\pi_\epsilon(t)).
\]

\noindent This minimum is always achieved at one of the finite set of corner weights and is always non-positive, since  $H_\epsilon^\pi(0)=0$.  We let $\mathcal P_{int}$ be the set of paths for which this 
$m_\epsilon$ is an integer for all $\epsilon \in I$. Littelmann proves in \cite{L}, Lemma 4.5(d), that all the Littelmann paths in the crystal for a dominant weight $\Lambda$ lie in $P_{int}$. Note that since $\pi(1) \in P(\Lambda)$, all the numbers $H_\epsilon^\pi(1)$ are integers. 

\begin{defn} Littelmann's function $ f_\epsilon$ is given on $\mathcal P_{int}$ as follows:
\begin{itemize}
\item If $H^\pi_\epsilon(1)=m_\epsilon$, then $f_\epsilon(\pi)=0$.
\item Set
\begin{align*}
t_0=&\max_t \{ t \in [0,1] \mid H^\pi_\epsilon (t)=m_\epsilon \}\\
t_1=&\min_t \{ t \in [t_0,1] \mid H^\pi_\epsilon (t)=m_\epsilon+1  \}
\end{align*}
then if $H^\pi_\epsilon(t)$ is monotonically increasing on the interval $[t_0,t_1]$, we define
\[
f_\epsilon(\pi)(t)=
\begin{cases}
\pi(t) &t \in [0,t_0]\\     \pi(t_0)+s_\epsilon(\pi(t)-\pi(t_0)) &t \in [t_0,t_1]\\\pi(t)-\alpha_\epsilon &t \in [t_1,1]
\end{cases}
\]
In the more complicated case, where the path  $H^\pi_\epsilon(t)$ is not monotonic  from $t_0$ to $t_1$, but contains some segments which oscillate, can be found in \cite{L}. We will have oscillating sections of the path, but they will all occur before $t_0$ or after $t_1$.
\end{itemize}
\end{defn}
\noindent The definition of $e_\epsilon$ is dual and reverses the action of $f_\epsilon$. The definition is given in \cite{L}. We will be using primarily $f_\epsilon$.

 Let $\mathcal P$ denote the set of paths. For $\nu,\mu$ defect $0$ weights, we define $\nu \geq \mu$ if there is a sequence of defect $0$ weights 
$\nu=\nu_0,...,\nu_s=\mu$ and positive real roots $\beta_1,\dots,\beta_s$ such that  
\[
\nu_i =s_{\beta_i}(v_{i-1})
\]
and
 \[\la \nu_{i-1},\beta_i^\vee \ra <0, i=1,2,\dots,s.
\]

Following a suggestion of Kashiwara, Littelmann then defines a distance function between $\nu$ and $\mu$.  We do not need the entire function and its properties, only the case $dist(\mu,\nu)=1$, which means that no other weight can be inserted between them preserving the order.

\begin{defn} \cite{L} An $a$-chain for $(\mu,\nu)$ is a sequence $\mu=\lambda_0>\lambda_1>\dots>\lambda_s=\nu$ of defect $0$ weights such that either $s=0$ and $\mu=\nu$ or $\lambda_i=s_{\beta_{i}}(\lambda_{i-1})$ for some positive real roots  $\beta_1,\dots,\beta_s$  with dist$(\lambda_i,\lambda_{i-1})=1$ and 
$a\la \lambda_i ,h_{\beta_i} \ra \in \mathbb Z$
\end{defn}

\begin{defn} \cite{L} A  \textit{rational path}  of class $\Lambda$ is an ordered set of sequences $(\underline \nu; \underline{h})$   for which  following hold:
\begin{itemize}
\item  $\nu_1>\dots>\nu_s $ is a linearly ordered set of defect $0$ weights in $W\Lambda$. 
\item  $0<b_1<b_2<\dots<b_s=1$ are rational numbers.  We will denote $0$ by $b_0$.
\end{itemize}
We identify $\pi$ with the parameterized path 
\[\pi(t)=\sum^{j-1}_{\ell=1} (b_\ell-b_{\ell -1})\nu_\ell +(t-b_{j-1})\nu_j,  b_{j-1} \leq t \leq b_{j}
\]
\end{defn}
\begin{defn}  A rational path $(\underline \nu, \underline b)$ is called an LS-path if there is an $b_i$-chain connecting $\nu_i$ to $\nu_{i+1}$.
We will generally write the rational path in the form 

\[
(\nu_1,\dots, \nu_s ;b_1,b_2,\dots,b_s=1).
\]
\noindent and call it  an LS-representation of the path. In this paper, the numbers $b_i$ will be endpoints of subintervals of $[0,1]$ parameterizing straight subpaths and will be called parameter endpoints.
\end{defn} In Example \ref{gaps}, we give an example of a rational path with non-trivial $b_i$-chains.  In \cite{AKT} the authors refer to $\nu_1$ as the ceiling and $\nu_s$ as the  floor.  Since there is a unique multipartition for a defect $0$ weight, they sometimes use the same notation for the corresponding multipartitions but in this article the ceiling and floor will always be weights.

\begin{defn} Let $B=B(\Lambda)$ be the crystal of a dominant integral weight $\Lambda$. For any $v  \in B$, we let $\theta(v)=(\theta_0,\dots, \theta_{e-1})$ be the hub of $v$, where 
\[
\theta_i=\la wt(v),\alpha_i^\vee \ra
\]
\noindent The hub is the projection of the weight of $v$ onto the subspace of the weight space generated by the fundamental weights \cite{Fa}.
\end{defn}

By doing numerous calculations of Littelmann paths, we found one particular class, which we call unidirectional,  that was particularly easy to handle. 
We consider the monotone sequence of weights of our rank $2$ affine Lie algebra, $\psi_1=s_0\Lambda$,   $\psi_2=s_1s_0\Lambda$, \dots  $\psi_m=\dots s_0s_1s_0\Lambda$, where $m$ is the number of reflections. For each $m$ we define an integer $d_m=rm-b$, which will be the number of rows in the multipartion corresponding to that weight. We will demonstrate later that the hub $\theta_m$ of $\psi_m$ is $[-d_m, d_{m+1}]$ if $m$ is odd and $[d_{m+1},-d_m]$ if $m$ is even.

\begin{defn} A Littelmann path for type A and rank $2$ will be called \textit{unidirectional}  if there is an integer $k \geq 0$ and a strictly descending sequence of positive integers of the same parity, $m_1,m_2,\dots m_{k}$,  a sequence of positive integers, $I_{m_1}, I_{m_2},\dots I_{m_{k}}$, an integer $m_{k+1}$ with $0 \leq m_{k+1}<m_k$,  and a non-negative sequence of integers $C_1,C_2,\dots, C_{k+1}$, such that  the projection of the path onto the hubs consists of two parts:
\begin{itemize}
\item The \textit{main} part, which is entirely contained in the second or fourth quadrant.  Let $\epsilon$ be the residue corresponding to the positive coordinate of the quadrant.  Its hubs are $\theta_m$ for $m_{k+1} \leq m \leq m_1$. so that $\epsilon =1$ in the second quadrant and $\epsilon=0$ in the fourth quadrant. Starting at the origin, the main part consists of $k$ \textit{long paths} of  positive integral $\epsilon$-length $I_{m_i}$, which are multiples of $\theta_{m_i}, i \leq k$, followed by a sequence of paths of the form $\frac{C_i}{d_{m}d_{m+1}}\theta_m$, for $m=m_i,m_i-1,\dots,m_{i+1}+1$, which we call \textit{oscillating paths}. Note that, if $C_i \neq 0$ so that the set of oscillating paths is non-empty, the first of the oscillating paths is a continuation of the long path with index $i$.
\item The \textit{seed} part, whose hub may contain or not contain any of the following:
\begin{itemize}
 \item A \textit{transit path}  $\frac{1}{d_{m_{k+1}+1}}\theta_{m_{k+1}}$, which may be embedded in a longer path.
\item A tail, which is a multiple of $\theta_0$ or of $\theta_1$
\item A long path which is a multiple of $\theta_1$ or $\theta_2$, , together with an oscillating path.
\item An oscillating set of straight paths, multiples of $\theta_u$ for $u \leq m_{k+1}$ with coefficient $\frac{C_{k+1}}{d_{u}d_{u+1}}$, where $C_k=C_{k+1}+d_{m_{k+1}}$.                                                              
\end{itemize}
In our main theorem, we will determine the exact significance of these elements of the seed part, but the intention of the name is that the transit path and the tail are the seeds from which new long paths  can grow.
\end{itemize}
\end{defn}

\section{MULTIPARTITIONS}

We now turn to the second labelling which will concern us in this paper, the multipartitions.  Unlike the Littelmann paths, this labelling is available only is the important case of type A.  Also unlike the Littelmann paths, it requires some non-canonical choices, a choice between two dual versions of the theory and the choice  of an object called the multicharge. Once these choices have been made, the aim of the paper will be to find conditions on the multipartition which insure that the corresponding Littelmann path will be unidirectional. We will then be able to read off the structure of the Littelmann path directly from the multipartition.

A \textit{multicharge} will be a sequence $s=(k_1,k_2,\dots,k_r)$ of integers with $0 \leq k_i \leq e-1$, for some natural number $r$ which will be called the $level$ of the multicharge. 

A multicharge $s$ determines a dominant integral weight
\[
\Lambda=\Lambda_{k_1}+\dots+\Lambda_{k_r}.
\]
\noindent In the abelian weight space the order of the summands is of course irrelevant; it is simply a notational reminder of the multicharge. We will follow Mathas in \cite{M} in requiring $k_1 \leq k_2 \leq \dots \leq k_r$.  We can then summarize by setting 
\[
\Lambda=a_0 \Lambda_0+\dots+a_{e-1}\Lambda_{e-1}
\]

 We can regard the integers $k_i$ as  elements of $\mathbb Z / e \mathbb Z$, and with a slight abuse of notation we will identify $k_i$ with this residue. 
We have to assume that we know the ceiling and floor of an multipartition, constructed  in \cite{AKT} from an LS-representation by letting the ceiling be $\nu_1$ and the floor be $\nu_s$.

 The Young diagram of the multipartition of a defect $0$ weight $\lambda$ will be represented by $Y(\lambda)$. If the $t$-th partition $\lambda^t$  of $\lambda$ is nonempty, then we associate to each node in the Young diagram a residue, where the node $(i,j)$ is given residue
\[
k_t +j-i
\]
This will be called a $k_t$-corner partition.  Due to the Mathas condition on the ordering of the $k_i$, this means that the multipartition will consist of $a_o$ $0$-corner partitions, followed by $a_1$ $1$-corner partitions and so forth.

\subsection{Standard LS-paths}

The main theorem of this paper will be a proof that a strongly residue homogeneous multipartition corresponds to a unidirectional Littelmann path. However, we do more:  we show that the LS-repesentation of the Littelmann path can be written down  directly from the Young diagram of the multipartition. This close connection between the multipartiton and the Littelmann path will be crucial in carrying out the induction in  our main theorem. An LS-representation for which both the defect $0$ weights and the parameter endpoints can be determined from the multipartition in the way we describe  will be called standard.   We will give the definition in general for any rank, but will apply it only in the case $e=2$.   We must first prepare considerable notation.

 An important part of the description of each partition $\lambda^t$ in the multipartition $\lambda$ is the length $a(\lambda^t)$ of the first row. Since these lengths decrease from the ceiling to the floor, the ascending numbering from ceiling to floor used in 
  \cite{K2} or \cite{AKT} is not compatible with our notation, so we will reverse the numbering.  We  assume that  we have fixed a  reduced word $s_{i_q}\dots s_{i_1}$ and the sequence of defect $0$ weights $\mu=\mu_q, \mu_{q-1},\dots, \mu_{0}$, with $\mu_j=s_{i_j}\mu_{j-1}$,  going back to $\Lambda$. By the partial ordering of defect $0$ weights which we introduced above, we thus have $\mu_j > \mu_{j-1}$ and in fact, since we are using simple roots, they  will have $dist(\mu_j , \mu_{j-1})=1$. That will be important, because our $a$-chains will be strings of these $\mu_j$.

 If we let $d_j$ be the number of times we must operate on $\mu_{j-1}$ with $f_{i_j}$ to produce the reflection $s_{i_j}$, then in terms of the multipartitions, we have 
\[
d_j=\#(Y(\mu_j)-Y(\mu_{j-1}))
\]
that is to say, the number of nodes in the difference between the two Young diagrams. The $d_j$ increase as $j$ increases.  Since all the added nodes have the same residue, we never add more than one to any given row. Now we consider the set
\[
((Y(\mu_j)-Y(\mu_{j-1}))\cap Y(\lambda))
\]
which is called a \textit{ladder} in the article by Fayers \cite{Fa2} generalizing the LLT-algorithm \cite{LLT} for canonical bases. If the ladder intersects the top row of the multipartition at node $m$ of residue $i$, we will call this the $m$-ladder, whereas for Fayers it was an $i$-ladder.  Not only the proof but even the statement of our main theorem will depend on the lengths of these ladders, which we denote by
\[
c_m=\#(((Y(\mu_j)-Y(\mu_{j-1}))\cap Y(\lambda)).
\]
Clearly $c_j \leq d_j$ for all $j$.  However, the $c_i$ do not necessarily increase as $j$ increases.

\begin{defn} Given a choice of reduced word and, as above,  a sequence  $\mu_q, \mu_{q-1},\dots, \mu_{0}$ of defect $0$ weights  set
\[
e_j=\frac{c_j}{d_j} 
\]
for $j>0$ and set   $e_0=1$. Let 
\[
j_p, j_{p-1},\dots,j_0
\] 
 be the subsequence of $q, q-1,\dots,0$ of $j$ for which  $e_j \neq e_{j+1}$
 so the $\mu_{j_0}$ will be the floor.  Then the LS-representation
\[
(\mu_{j_p},\dots, \mu_{j_0};e_{j_p},\dots,e_{j_0})
\]
will be called \textit{standard} with respect to the multipartition $\lambda$ if the sequence of $e_i$ is increasing and the Littelmann path determined by this LS-representation corresponds in the crystal to the multipartition.
\end{defn}

\begin{examp}(A standard LS-representation with gaps)\label{gaps}
Let $e=3$. let $\Lambda=3\Lambda_0$  and set $\lambda =[(8,6,1),(),()]$. From our computer calculations, we know that the ceiling is the defect zero  multipartition with  weight 
$-12\Lambda_1 +15\Lambda_2-24 \delta$. This  corresponds to three copies of the partition $(8,6,4,2)$, and equals three times the weight of a single copy of that partition in $B(\Lambda_0)$. This ceiling has a periodic representation as $s_1s_0s_2s_1s_0s_2s_1s_0\Lambda$. The first ladder is the $8$-ladder with residue $1$, consisting of positions $8$ and $6$ in the first two rows.  For $m=7$, we also get an $m$-ladder of length $2$, consisting of positions $7$ and $5$ in the first two rows, and similarly for $m=6$,  we get a ladder of length $2$, so that $c_8=c_7=c_6=2$. Now, however, with the $5$-ladder of residue $1$ we get $c_5=3$. 

 If we calculate the $d_j$, descending from $8$ to $1$, we get $12, 12,9,9,6,6,3,3$.  The 
$c_j$ in the same order are $2,2,2,3,2,2,1,1$. The quotients are then $\frac{1}{6},\frac{1}{6}, \frac{2}{9}, \frac{1}{3}, \frac{1}{3}, \frac{1}{3}, \frac{1}{3}, \frac{1}{3}$. Finally, we get $e_0=1$ with weight $\Lambda$. Since some of the adjacent fractions are identical, this means that we do not need the entire sequence $\mu_8,\mu_7,\dots,\mu_1$ but that, in writing a standard LS-representations, the $\mu_j$ corresponding to differences which are equal to zero should be omitted.  
 When we eliminate all $j$ for which $e_j=e_{j+1}$, 
we are then left with $\mu_8,\mu_6,\mu_5,\Lambda$ as the weights, and quotients  $\frac{1}{6}, \frac{2}{9}, \frac{1}{3},1$ as the fractions. Since we take difference between the fractions, the actual Littelman path is
\[
\frac{1}{6}\mu_8, \frac{1}{18}\mu_6, \frac{1}{9}\mu_5,\frac{2}{3}\Lambda,
\]
\noindent with the coefficients summing to $1$ as they should.
\end{examp}

We are dealing, in this paper, with the $e$-regular case, in which addable node are added from the top right down to the bottom left, as distinguished from the
 $e$-restricted case in \cite{Kl} and \cite{AKT} where the addable nodes are added from the bottom left to upper right.  In trying to reconstruct the multipartition in a non-recursive fashion from the Littelmann path in the form we have written it, all we have is the number of nodes of $Y(\lambda)$ intersecting $(Y(\mu_j)-Y(\mu_{j-1}))$, not their location.  However, if they are fact located from the top down, then we could indeed reconstruct $\lambda$ if we knew the reduced word, giving us the $d_m$,  and the fractions $e_m$ correponding to the various $\mu_j$, which would give us the $c_m $ as $d_me_m$.

\section{THE CASE $\Lambda=a\Lambda_0+b\Lambda_1$}

We consider the case $e=2$, with multicharge $s=(0,0,\dots,0,1,1,\dots,1)$.   If there are $a$ copies of $0$ and $b$ copies of $1$, with level $r=a+b$, then when we fill in residues in the nodes of the Young diagram of a multipartition with this multicharge, there will be $a$ of the  $0$-corner partitions followed by $b$ of the
$1$-corner partitions.
   A partition in which the parts  alternate between odd and even will be called \textit{alternating}.
\begin{defn} A multipartition for $e=2$ will be \textit{residue-homogeneous} if 
\begin{itemize}
\item Every nonempty partition is alternating.
\item All $0$-corner partitions have the first line of the same parity, and the $1$-corner partitions have the first line of the opposite parity. The last row of every non-zero partition, except possibly the last, is odd.

\end{itemize}
\end{defn}

\begin{defn} A multipartition for $e=2$ will be \textit{strongly residue-homogeneous} if 
\begin{itemize}
\item It is residue homogeneous, and
\item  For each partition $\lambda^i$ after the first partition, if the top row has length $n$, the previous partition ends with a triangle of $n-1$ rows and columns if both have the same corner residue, and if it is the top partition of the  the $1$-corner partitions, then the previous partition ends with a triangle of length $n-2$ if it has fewer rows than  $\lambda^i$, and of $n$ if it has $n$ or more rows.                                                                                                                                                                                                                                                                                       
\end{itemize}
\end{defn}

We recall that, by the dual of Mathas' result  Prop. 4.9  in \cite{M}, a multipartition is $2$-regular if and only if  the length $a(\lambda^t)$  of the first row of a partition is  less that or equal to the number of rows  $\ell(\lambda^{t-1})$ in the previous partition, except on the boundary between $0$ and $1$, where, for words beginning with $s_1$, it can be greater by $1$.  If the top row of the first $1$-corner partition is not greater by $1$, we say that the multipartition is non-increasing.

\begin{lem}
A strongly residue-homogeneous multipartition is $2$-regular.
\end{lem}

\begin{proof}
We begin with two adjacent partitions $\lambda^{t-1}$ and $\lambda^t$ in the multipartition which have the same corner residue, either $0$ or $1$. By the definition of residue homogeneous, the top rows have the same parity.  Letting $n=a(\lambda^t)$, the length of the top row of the lower partition, then by the condition that the multipartition be \textit{strongly} residue homogeneous, the upper partition ends in a triangle with at least $n-1$ rows.  Thus $\ell(\lambda^{t-1}) \geq n-1$.
Because $\lambda^{t-1}$ is a partition with no two rows identical, $a(\lambda^{t-1}) \geq \ell(\lambda^{t-1})$. By the condition that the parities are the same, we get that $\ell(\lambda^{t-1}) \geq n$, which is the first condition needed for the dual Mathas result quoted above the lemma. At the boundary between $0$-corner partitions and $1$-corner partitions, we no longer have the parities equal, but we get as before that  $\ell(\lambda^{t-1}) \geq n-1$, and this is enough.
\end{proof}

 For the case of $e=2$, there are only two families of reduced words.  The words with $s_0$ first on the right, i.e., $\dots s_1s_0$,  give the defect $0$ weight whose multipartitions have the upper triangular partitions larger  will be denoted by   $\psi^-_m, m>0$. The words with $s_1$ on the right, i.e. $\dots s_0s_1$, give the defect $0$ weights  $\psi^+_m, m\geq 0$ whose upper triangular partitions are smaller.
 All the $0$-corner partitions have the same 
first row $m$, and all the $1$-corner partitions all have the same length first row, which is either $m-1$ or $m+1$, respectively. When $b=0$, we will use $\psi_m^-$ by default.

The floor and ceiling in the case $r=1$ are determined by the length  $\ell(\lambda)$ and  $a(\lambda)$ of the first column and first row, respectively.  The situation for $r>1$ is considerably more complicated. 
In the case of general $e$ for which we defined the standard LS-representation, the denominator $d_m$ depended on the choice of reduced word. 

We now define $d^\pm_n$ to be the number of rows in a defect $0$ multipartition corresponding to a word beginning on the right  with $s_1$ or $s_0$, respectively, where $n$ is the length of the first row.  Since all partitions with $0$ in the corner have $n$ rows, and all partitions with $1$ in the corner have $n \pm 1$ rows, we get
\[
d^\pm_n=an+b(n \pm 1)=rn \pm b=r(n \pm 1)\mp a
\]
The corresponding projection of the weight to the hub is 
\[
\theta^\pm_{2s}=[d^\pm_{2s+1},-d^\pm_{2s}]
\]
\[
\theta^\pm_{2s+1}=[-d^\pm_{2s+1},d^\pm_{2s+2}]
\]
In order to avoid writing every hub twice for the odd and even cases, we let $\phi$ be the operation of interchanging the $0$ and $1$ coordinates and then can write 
\[
\theta^\pm_{m}=\phi^m[d^\pm_{m+1},-d^\pm_{m}]
\]
\noindent Note that when $a=1$ and $b=0$, $d_m^-=m$.

\begin{defn} The \textit{segments} are defined by putting together rows whose lengths drop by only one, with the following exception:
if the length of the first row of the partition equals the number $\ell$  of rows in the previous partition, or, at the boundary between $0$ and $1$, equals $\ell \pm 1$, then it does not start a new segment. The segment boundary is  the line between the last row of the old segment and the first of the new. For strongly residue-homogeneous multipartitions, the segments will correspond to the long paths of the corresponding unidirectional Littelmann path.
\end{defn}

\begin{defn}
We number the segments from $1$ to $k$, and for segment $i$  
\begin{itemize}\label{segnot}
\item We let $v_i$ be the number of partitions before the end partition of the segment, where $v_i=0$ if the segment ends in the first partition.
\item  Inside the last partition intersecting the segment, we let $t_i$ be the number of  rows inside the partition down to the last row of the segment.
\item We let $z_i$ be a Boolean parameter equal to $1$ if the segment ends in a singleton and $0$ if not. 
\item Let $n_i$ be the length of the first row of the segment as before.  If the segment starts at the top of a $0$-corner partition, then set $n_i'=n_i$, and at the top of a  $1$-corner partition, $n_i'=n_i \mp 1$.  If it starts in the middle of  a $0$-corner partition, set $n_i'=n_i +t_{i-1}$ and in the middle of a $1$-corner partition, $n_i'=n_i+t_{i-1}\mp 1$.  

\end{itemize}
\end{defn}
 Because the partitions are alternating and the first rows of $0$-corner partitions all have the same parity opposite to that of the $1$-corner partitions, all the $n_i'$ will have the same parity.  We also have that 
$n_1'>n_2'> \dots> n_k'$. At each segment boundary, we extend the bottom row of segment $i$ upward to the top of the partition, which it will hit at $n_i'$, and this will be less than the point $n_{i-1}'$ at which the segment $i-1$ will hit the top of the partition, since otherwise the two segment would have been combined into one.
\begin{remark}\label{slice}
The segment corresponds to  a number of adjacent rows in the multipartition with weight $\psi_{n_i'}^-$ or  $\psi_{n_i'}^+$. If $\lambda$ is a non-increasing partition, we will use ``-'' and if it is an increasing partition at the boundary between $0$-corner and $1$ corner partitions, we use $\psi_{n_i'}^+$.  The default, when the multipartition does not contain non-empty $1$-corner partitions, will be ``-''.
\end{remark}

The proof of our main theorem, that a strongly homogeneous multipartitions has a unidirectional Littelmann path, involves a double induction.  The main induction is on the length $n$ of the top row of the multipartition, but there will be a further internal induction on the segments.  In the course of this second induction we will need to know not only the lengths of various ladders, but also the lengths of the intersections of the ladders with the various segments. In order to calculate this precisely, we need more notation.

 Letting $m$ be a number satisfying $n_{i+1}'<m \leq n_i'$,  we set $v^m=v_i$, the number of partitions above the partition containing the last row of the segment. 

Our aim is to find a formula for the  length of the $m$-ladder that contains the  $m$th-node in the top row of the multipartition.  The integer $v^m$ will allow us to find the length of the intersection of the $m$-ladder with the partitions above the end partition, and now we must define a new quantity to measure the length of the part of the $m$-ladder in the last partition of the segment.  If the $m$-ladder ends on the last row of the segment, then this length is $t_i$, and particularly if  $m=n_i'$,  the number of rows in the last partition intersection the segment.  If the $m$-ladder ends in the first column, then the length of its intersection is $m$ or $m \pm 1$, depending on whether we are in a $0$-corner partition or a $1$-corner partition.  In either case we get a number less than or equal to $t_i$.  Therefore we define

\[
t^m=
\begin{cases}
\min(t_i,m), &v_i < a\\
\min(t_i, m \pm 1), &v_i \geq a
\end{cases}
\]
for each such $m$, where we use the option $``-"$ for non-increasing multipartitions.
We can combine the two cases as 
\[\min(t_i, m \pm 1_{v_i \geq a}).\]
 The integer $t^m$ was defined to give exactly the number of nodes in the intersection of the ladder with the last partition.

 We can calculate $d^\pm_m$ as above, and we can now give a algebraic formula for the lengths of the $m$-ladders. 

\[
c_m=
\begin{cases}
v^m\cdot m +t^m,&v^m \leq a\\
a \cdot m +(v^m-a) \cdot (m \pm 1) +t^m &v^m>a
\end{cases} 
\]

If we use the standard function $x_+=\max(x,0)$, then we can unite the two cases into a single formula
\[
c_m=v^mm\pm(v^m-a)_++t^m.
\]

If this is in segment $i$, we get
\[
c_m=v_im \pm(v_i-a)_++\min(t_i, m \pm 1_{v_i \geq a}).
\]

 The value of $c_m$ will change in the middle of the segment $i$ whenever $v_i>0$ and also if $t^m$ depends on $m$.  We always have $t^m=t_i$ when $m=n_i'$, but if the segment ends in a singleton or if $i=k$ and we pass to the seed part of the path, we can have  $t^m<t_i$.   
As in the introduction, we set
\[ e_m=\frac{c_m}{d_m}
\]

In the  induction step of the main theorem, we will be required to prove that each segment corresponds to a straight path of integral length in the positive coordinate of the hub, whose length equals the number of addable nodes in that segment.  We will be required to express this integral  number in terms of various numbers associated with the multipartition, which we now define.
\begin{defn}

In order to deal with cases $n_{i+1}'<m<n_i'$, we extend the Boolean parameter $z_i$ from Def. \ref{segnot}  by setting $z^m$ equal to $0$ if the intersection of the segment with $\psi^\pm_m$ does not  end with a singleton, and equal to $1$ if it does.

To deal with the last segment, we let $n'_{k+1}=0$ if it ends in a singleton, i.e., $z_k=1$.  if the last row is longer than $1$, we prepare the way for starting a new segment at the end of the current partition by setting $z_{k+1}=1$ if the last row is odd and  $z_{k+1}=0$ if the last row is even,   and set   $n'_{k+1}=t_k \mp 1_{v_i \geq a},n_{k+2}=0.$   When $m=n_{k+1}'$, let $z^m =z_{k+1}$.Finally, 
define
\[
I_m=c_m-c_{m+1}+v^m+z^m.
\]
\end{defn}

\begin{lem}\label{ri} We divide into cases:
\begin{itemize}
\item If  $m=n_i'$ for $i \leq k$, then $I_m=c_m-c_{m+1}+v^m+z^m$ is the number of addable rows in the segment $i$, with  
\begin{align*}
I_{n_i'}&=(v_i-v_{i-1})(n_i'+1) \pm ((v_i-a)_+-(v_{i-1}-a)_+)\\
&+(t_i+z_i-\min(t_{i-1},n_i'+1 \pm 1_{v_i \geq a}))
\end{align*}
\item If  $n_{i+1}'<m<n_i'$, for $1 \leq i \leq k+1$, then $I_m=0$. 
\item If $m=n_{k+1}'>0$, then $I_m=z_{k+1}$. 
\end{itemize}

\end{lem}

\begin{proof}
If $m=n_i'$ for $i \leq k$, then $t^m=t_i, v^m=v_i, z^m=z_i, $ and  $c_{m+1}$ belongs to segment $i-1$.  We then get 
\begin{align*}
I_m&=c_m-c_{m+1}+v^m+z^m\\
&=(v_im \pm(v_i-a)_++t_i)\\
&-(v_{i-1}(m+1) \pm(v_{i-1}-a)_+)+\min(t_{i-1},m+1 \pm 1_{v_{i-1} \geq a}))+v_i+z_i.\\
&=(v_i-v_{i-1})(n_i'+1) \pm((v_i-a)_\pm-(v_{i-1}-a)_+)\\
&+(t_i+z_i-\min(t_{i-1},n_i'+1 \pm 1_{v_{i-1} \geq a}))
\end{align*}

It remains to show that this is the number of addable nodes in segment $i$. All the addable nodes of segment $i$ lie in the $n_i'+1$ ladder. The number of possible internal partitions of the segment is  $v_i-v_{i-1}$ .  This would give us 
$(v_i-v_{i-1})(n_i'+1)$ if they were all $0$-corner partitions, but to adust for the possibility that some of the partitions are $1$-corner partitions and are thus larger or smaller by $1$, we must add an adjustment term $\pm((v_i-a)_+-(v_{i-1}-a)_+)$. Now from the last partition, there are an additional $t_i+z_i$ addable nodes in the segment.  If the segment $i-1$ ends in middle of a partition, we have to remove $t_{i-1}$, since those nodes belong to the previous segment.  If the segment ends at a segment boundary, we need to remove only the nodes in the $n_i'+1$-ladder, of which there are $n_i'+1 \pm 1_{v_{i-1} \geq a}$.

If $n_{i+1}' <m<n_i'$, then $c_{m+1}$ belongs to  the same segment $i$, so $c_m-c_{m+1}=-v_i-z_i$, which means that $I_m=0$.

If  $n_{k+2}' <m \leq n_{k+1}$, then we are not in a real segment, but in a virtual segment, which will only become a real segment if we decide to add the one addable node in the next row when $z_{k+1}=1$ .  We have  $c_m-c_{m+1}=-v_k-z^{m+1}$, so $I_m=z^m-z^{m+1}$.  If $m<n_{k+1}'$ , then   $z_{m+1}=z^m=1$ so again $I_m=0$ in that case.

 If $m=n_{k+1}'>0$, then $z^{m+1}=0$ and $z^m=z_{k+1}$, and so $I_m=z_{k+1}$ and thus is either $0$ or $1$, depending on whether there is or is not an addable node.

\end{proof}

The integer $I_{n_i'}$ is the number of times we must operate by the appropriate $f_\epsilon$ in order to fill in the segment to the next level, and we are interested in having this as our denominator for the long paths.  
\begin{examp}
For $e=2, \Lambda=3\Lambda_0+2\Lambda_1$, the multipartition

 \[
 [(11,10,7,6,5,4,3,2,1),(7,6,5,4,3,2,1),(7,6,3,2,1),  (2,1), \emptyset]
\]
\noindent has five segments, containing, $2, 7,9,3,2$ rows where the respective segment pairs $(v_i,t_i)$  are 
$(v_1,t_1)=(0,2), (v_2,t_2)=(0,7), (v_3,t_3)=(2,2),(v_4,t_4)=(2,5), (v_5,t_5)=(3,2) $.

Since $c_{12}=0$ and $v^{11}=z^{11}=0$, we get $I_{11}=c_{11}=2$. For the second segment, we have $n_2'=9$, so $I_9=c_9-c_{10}+v^9+z_1=9-2+0+1=8$. For the third segment, we use $c_7=16, c_8=8, v^8=2$, so $I_7=16-8+2=10$ .  The fourth segment has $3$ rows, and we get  $I_5=c_5-c_6+v^5+z_4=15-14+2+1=3$. In the fifth and last segment we have by $I_2=c_2-c_3+v^2+z_5=8-9+3+1=3$. All the remaining $I_m$ are zero. The last row ends in a singleton, so there is no virtual segment $k+1$.

\end{examp}

\begin{lem}\label{built} Every $2$-regular non-increasing strongly residue-homogeneous multipartion $\lambda$ can be built up from the empty partition by a sequence of operations  of the following two types:
\begin{enumerate} 
\item (Widening) Adding every addable node down to the segment boundary of the last non-zero  segment, giving $\lambda^*$.
\item (Deepening) Adding every addable node down to the end of $\lambda$ and then one singleton, in the same partition if the last row is odd and in the next partition if the last row is even and not at the boundary from $0$-corner to $1$ corner residues, to get $\bar \lambda$ and then, if desired, adding singletons of the same corner  residue to some of the addable nodes in  previously empty partitions to get $\tilde \lambda$. A  deepening operation will always be followed by a widening or by a deepening which does not add a new partition with the same residue in the corner.
\end{enumerate}
\end{lem}

\begin{proof}
Induction on the number $m$  of nodes in the first line of the multipartition $\lambda$.   Assume that the lemma is true for every  strongly residue-homogeneous multipartition with first line less than $m$. If the last row  of the last non-zero segment  is  not a singleton, then removing  a  node from each row above the segment boundary is a reversible operation whose inverse is a widening, producing a strongly residue-homogeneous multipartition with first row $m-1$, to which we can apply the induction hypothesis. The second condition in the definition of  strongly residue homogeneous is stable because, though we reduce the triangles at the ends of each partition  except the last by removing one node from each row, we also reduce the first row of each partition by one node.

 If the last non-zero row is a singleton, then we have two cases.  If it is the last row of a larger partition, we apply a reversible operation whose inverse is a deepening, removing one node from each non-zero row. The last non-zero row will vanish. The number of rows in the triangular tails will be reduced by $1$, but so will the first row of each partition.  We are reduced to a strongly residue-homogeneous multipartition with top row $m-1$ and we apply the induction hypothesis.

We will call the partition $(1)$ a solitary partition. If the singleton is a solitary partition, then we must remove all the solitary partitions above it.  If it is in the $1$-corner partitions, then this process must stop at or before the boundary between $0$-corner and $1$-corner partitions, because at that point the residue of a solitary partition changes from $1$ to $0$, and our multipartition is residue homogeneous. If it is a $0$-corner partition, then if all the partitions are solitary, we have gotten back to the empty partitions.   We now consider two subcases.

\begin{itemize}
\item Suppose that the corner of the upper solitary partition and the last non-solitary partition have the same corner residue. If the last row of this lowest non-solitary partition is a singleton, then we remove one node from each row and have a deepening.  If it is not a singleton, then by the definition of strongly
 residue-homogeneous, it must be odd and removing one node from each row  will give an even last row and we will have a deepening which adds partitions.

\item Suppose the corner of the  solitary partition has corner residue $1$ and  the last non-solitary partition has corner residue $0$.  Then by the definition of strongly residue-homogeneous the last $0$-corner partition must end in a singleton, and we continue by the reverse of a deepening.
\end{itemize}

\end{proof}

Our aim is to show that we have  long paths alternating with oscillating short paths. The lengths of the long path corresponding to segment $i$ will correspond to the number of $\epsilon$-addable nodes in the segment.  The hub of the long path corresponding to segment $i$ is a multiple of $\theta_m$, where $m=n_i'$. The positive coordinate of $\theta_m$ has value $d_{m+1}$ and, in a standard Littelmann path, the coefficient is $e_m-e_{m+1}$. We begin with a technical lemma, which, when combined with Lemma \ref{ri}, will do most of the work needed to demonstrate unidirectionality.

\begin{lem}\label{q} Let $\lambda$ be a $2$-regular strongly residue-homogeneous multipartition with standard LS-representation. 
Let $m$ be an integer such that   $n_{i+1}'<m \leq n_i'$. The coefficient $q^\pm_m=e_m-e_{m+1}$ of the hub  $\theta^\pm_m$ in the Littelmann path can be written in the form 
\[
q_m^\pm=\frac{I_m}{d^-_{m+1}}+\frac{C_i}{d^-_md^-_{m+1}}
\]
\noindent where $I_m =0$ unless $m=n_i'$, in which case
\[
I_{n_i'}=(v_i-v_{i-1})(n_i'+1)-((v_i-a)_+-(v_{i-1}-a)_+)+(t_i+z_i-\min(t_{i-1},n_i'+1-1_{v_{i-1} \leq a}))
\]
and
\[C_i=
\begin{cases}
\mp v_ib\pm r(v_i-a)_++rt_i, &z_i=0, i \leq k, \\
\mp (v_i+1)b\pm r(v_i-a)_+\pm r1_{v_i \geq a}, &z_i=1,~ or~i=k+1
\end{cases}
\]
\end{lem}

\begin{proof}Before deriving the formulae, we note now that we will substitute $d^\pm_{m+1}=d^\pm_{m}+r$  in the numerator  but not in the denominator.  We will also use the formula $d_m^\pm=rm \pm b$.
By definition,
\begin{align*}
 q^\pm_m=&e_m-e_{m+1}\\=&\frac{c_m}{d^\pm_m}-\frac{c_{m+1}}{d^\pm_{m+1}}\\
=&\frac{c_md^\pm_{m+1}-c_{m+1}d^\pm_{m}}{d^\pm_md^\pm_{m+1}}.\\
=&\frac{(c_m-c_{m+1})d^\pm_{m}+rc_{m}}{d^\pm_md^\pm_{m+1}}.
\end{align*}
Now  $I_m=c_m-c_{m+1}+v^m+z^m$ is the numerator we want for $d_{m+1}^\pm$.  In order to get it, we must add and subtract  $(v^m+z^m)d_m^{\pm}$, getting
\begin{align*}
 q_m^\pm=&\frac{(c_m-c_{m+1}+v^m+z^m)d^\pm_{m}}{d^\pm_md^\pm_{m+1}}+\frac{-(v^m+z^m)d^\pm_{m}+rc_{m}}{d^\pm_md^\pm_{m+1}}\\
=&\frac{I_m}{d^\pm_{m+1}}+\frac{-(v^m+z^m)(rm \pm b)+rc_{m}}{d^\pm_md^\pm_{m+1}},\\
=&\frac{I_m}{d^\pm_{m+1}}+\frac{-v^mrm-z^m rm \mp v^mb \mp z^m b+r(v^mm+t^m\pm(v^m-a)_+)}{d^\pm_md^\pm_{m+1}}\\
=&\frac{I_m}{d^\pm_{m+1}}+\frac{\mp v_ib\pm r(v_i-a)_+ \mp z^mb+r(t^m-z^m m)}{d^\pm_md^\pm_{m+1}},
\end{align*}
\noindent where in the last step we cancelled two copies of  $v^mrm$  with opposing signs and substitute $v_i$ for $v^m$.
We now divide into cases, since the values of $z^m$ and $t^m$ may depend on $m$.

\noindent Case 1:  $z_i=0$, $i<k$.

Then $z^m=z_i=0$ for all $m$ satisfying $n_{i+1}'<m \leq n_i'$, and $t^m=t_i$, so by substitution we get 
\[
C_i=\mp v_ib\pm r(v_i-a)_++rt_i
\]
\noindent Case 2:  $z_i=1$.

Then $z^m=z_i=1$ for all $m$ satisfying $n_{i+1}'<m \leq n_i'$, and $t^m=m \pm 1_{v_i \geq a}$, where the additional $1$ is added or subtracted only in the case $v_i \geq a$.  Then $t^m-z^mm=\pm 1_{v_i \geq a}$, so by substitution we get 
\[
 C_i=\mp (v_i+1)b\pm r((v_i-a)_+ +1_{v_i \geq a})
\]

\noindent Case 3:  $i=k+1$:   If $m  \leq n_{k+1}'$, we have $z^m=1, t^m=m\pm  1_{v_i \geq a}$ so then we get, as in Case 2, 

\[
 C_i=\mp (v_i+1)b\pm r((v_i-a)_++1_{v_i \geq a})
\]

\end{proof}
\begin{corl}
 For $i<k$, or $i=k$, and either $z_i=1$ or  $z_k=0$ and $m \pm 1_{v_i \geq a}> t_i$ we have 
\[
C_i=\mp (v_i+z_i)b\pm r(v_i-a)_++r((1-z_i)t_i \pm z_i1_{v_i \geq a})),
\]
\end{corl}
\begin{proof}
In the listed cases,  the choice of formula depends only on the value of $z_i$, so we take a weighted average with weights $z_i$ and $1-z_i$.
\end{proof}

\noindent The number $C_i$ will be called the\textit{ oscillator}, and the important point is that there are two kind of oscillator. The first, called the \textit{widening oscillator}, is for oscillating paths following a segment not ending in a singleton, or between the last segment and the transit path, while the second, which is independent of $t_i$, is only used for after segments which end in a singleton, where it is called a \textit{deepening oscillator}  or for  the oscillating paths after the (possibly empty)  transit path connecting to the (possibly empty) tail, where it will be called the \textit{terminal oscillator}.

When $a=1$ and $b=0$, so that we have a single partition, then $C_i=t_i$.  What makes the general case so much more complicated is that the addable nodes of segment $i$ lie not only at the ends of rows, but also at boundaries between partitions. 
This segment boundary adjustment will occur frequently in our calculations, so denote it by $s_i$, where
\[
s_i=v_i+z_i-v_{i-1}-z_{i-1}
\]  
\begin{lem}\label{dif}
 For $i \leq k$, we have 
\[
C_i-C_{i-1}=rI_{n_i'}-d_{n_i'+1}s_i,
\]
For $i=k+1$, we have 
\[
C_{k+1}-C_{k}=-d_{n_{k+1}'},
\]
If $i=k+2$, with $w$ non-zero partitions and  $y$ added partitions, then 
\[
C_{k+2}-C_{k+1}=
\begin{cases}
yb,&w\leq a,\\
-ya,& r>w>a,
\end{cases}
\]
Thus the terminal oscillator rises in steps of $b$ from $0$ to $ab$ until the primary  tail $\frac{a-w}{a}\theta_0$ is used up, then descends from $ab$ to $0$ is steps of $a=d_1$, in such a way that  the secondary tail, which  coincides with the oscillating path $\frac{C_{k+1}}{d_1d_2}\theta_1=\frac{r-w}{d_2}\theta_1$  in $\bar \lambda$, becomes
$\frac{C_{k+2}}{d_1d_2}\theta_1=\frac{r-w-y}{d_2}\theta_1$ in  $\tilde \lambda$.
\end{lem}
\begin{proof}
The right side of the equation  contains a large term equal to $r(m+1)(v_i-v_{i-1})$ which occurs twice with opposite signs and must be cancelled, so we will begin with the right side of the equation. 
\begin{align*}
rI_m-d_{m+1}s_i&=r\left [(v_i-v_{i-1})(m+1)-((v_i-a)_+-(v_{i-1}-a)_+)\right ]\\
&+\left [(t_i+z_i-\min(t_{i-1},m+1-1_{v_{i-1} \geq a})\right]-(r(m+1)-b)s_i\\
&=r\left [-((v_i-a)_+-(v_{i-1}-a)_+)\right ]\\
&+\left [t_i+z_i-\min(t_{i-1},m+1-1_{v_{i-1} \geq a})\right]-(r(m+1)(z_i-z_{i-1}))+bs_i\\
&=r \left [-(v_i-a)_++(1-z_i)t_i+z_i1_{v_i \geq a} \right ]+(v_i+z_i)b\\
&+r\left [(v_{i-1}-a)_+-(1-z_{i-1})t_{i-1} -z_i1_{v_{i-1} \geq a}\right ]-(v_{i-1}+z_{i-1})b.\\
&=C_i-C_{i-1}
\end{align*}

In the $i=k+1$ case, we have $v_k=v_{k+1}$, and $t_k=n_{k+1}\pm -1_{v_{k+1} \geq a}$, so
\begin{align*}
C_{k+1}-C_{k}&= \left [(v_{k+1}+1)b \pm r(v_{k+1}-a)_+\pm r1_{v_{k+1} \geq a} \right ]\\
&+\left[ v_k b \pm r(v_k-a)_++rt_k \right ].\\
&=b\pm r1_{v_{k+1} \geq a}-rt_k \\
&=b\pm r1_{v_{k+1} \geq a}-r(n_{k+1}\pm 1_{v_{k+1} \geq a})\\
&=b-r(n_{k+1}')\\
&=-d_{n_{k+1}'}
\end{align*}

If $w$ is the number of non-empty partitions in $\lambda$, then from the formula for the second oscillator, substituting $w$ for $v_k+1$, we have that the terminal oscillator $C_{k+1}$ is
\[C_{k+1}=
\begin{cases}
wb,&w\leq a,\\
(r-w)a,& r>w>a,
\end{cases}
\]
The new terminal oscillator $C_{k+2}$ will be the same with $w$ replaced by $w+y$, all within the $0$-corner or $1$-corner partitions, since we are adding the $y$ solitary partitions with a single residue.

\end{proof}

\begin{examp} Consider $e=2$, $a=3$, $b=2$. We calculate the $q_m$ in this format for a sequence of multipartitions. The first has a transit path at $\theta_3$, with $d_3=13$ and $17=C_1=C_2+d_3$. The fourth has a transit path at $5$, with $d_5 =23$ and $27=C_2= C_3+d_5$. We will continue to have $C_{k+1}=4$ until we open a new partition and get terminal oscillator $6$ or $3$.
\begin{itemize}
\item $ [(5,4,3,2,1)(5,4,3), \emptyset,\emptyset,\emptyset$]\\
$e: \frac{8}{23}, \frac{7}{18}, \frac{6}{13}, \frac{4}{8}, \frac{2}{3},1,$\\
$  (\frac{9}{28}+\frac{17}{28\cdot 23})\theta_5, \frac{17}{23\cdot 18}\theta_4,  \frac{17}{18\cdot 13}\theta_3, \frac{4}{13 \cdot 8}\theta_2, \frac{4}{8\cdot 3}\theta_1,\frac{1}{3}\theta_0$
\item $ [(6,5,4,3,2,1)(6,5,4,1), \emptyset,\emptyset,\emptyset]$\\
$e: \frac{9}{28}, \frac{8}{23}, \frac{8}{18}, \frac{6}{13}, \frac{4}{8}, \frac{2}{3},1,$\\
$  (\frac{10}{33}+\frac{17}{33\cdot 28})\theta_6, \frac{17}{28\cdot 23}\theta_5,(\frac{2}{23}+\frac{4}{23\cdot 18})\theta_4,  \frac{4}{18\cdot 13}\theta_3, \frac{4}{13 \cdot 8}\theta_2, \frac{4}{8\cdot 3}\theta_1,\frac{1}{3}\theta_0$
\item $ [(7,6,5,4,3,2,1)(7,6,5,2,1), \emptyset,\emptyset,\emptyset]$\\
$e: \frac{10}{33},\frac{9}{28}, \frac{10}{23}, \frac{8}{18}, \frac{6}{13}, \frac{4}{8}, \frac{2}{3},1,$\\
$  (\frac{11}{38}+\frac{17}{38\cdot 33})\theta_7, \frac{17}{33\cdot 28}\theta_6,   (\frac{3}{28}+\frac{4}{28\cdot 23})\theta_5\frac{4}{23\cdot 18}\theta_4,   \frac{4}{13 \cdot 8}\theta_3, \frac{4}{13 \cdot 8}\theta_2, \frac{4}{8\cdot 3}\theta_1,\frac{1}{3}\theta_0$
\item $ [(8,7,6,5,4,3,2,1)(8,7,6,3,2), \emptyset,\emptyset,\emptyset]$\\
$e:  \frac{11}{38},\frac{10}{33},\frac{11}{28}, \frac{10}{23}, \frac{8}{18}, \frac{6}{13}, \frac{4}{8}, \frac{2}{3},1,$\\
$  (\frac{12}{43}+\frac{17}{43\cdot 38})\theta_8,\frac{17}{38\cdot 33}\theta_7,(\frac{2}{33}+ \frac{27}{33\cdot 28})\theta_6,  \frac{27}{28\cdot 23}\theta_5,\frac{4}{23\cdot 18}\theta_4,   \frac{4}{13 \cdot 8}\theta_3, \frac{4}{13 \cdot 8}\theta_2\\
, \frac{4}{8\cdot 3}\theta_1,\frac{1}{3}\theta_0$
\item $\lambda= [(9,8,7,6,5,4,3,2,1)(9,8,7,4,3), (1),\emptyset,\emptyset]$\\
$e:\frac{12}{43},  \frac{11}{38},\frac{12}{33},\frac{11}{28}, \frac{10}{23}, \frac{8}{18}, \frac{6}{13}, \frac{4}{8}, \frac{3}{3}$\\
$(  \frac{13}{48}+\frac{17}{48\cdot 43})\theta_9,\frac{17}{43\cdot 38}\theta_8,(\frac{2}{38}+\frac{27}{38\cdot 33})\theta_7,\frac{27}{33\cdot 28}\theta_6,  (\frac{1}{28}+\frac{4}{28\cdot 23})\theta_5,\frac{4}{23\cdot 18}\theta_4,   \frac{4}{13 \cdot 8}\theta_3, \frac{4}{13 \cdot 8}\theta_2,\\
(\frac{1}{8}+\frac{1}{8}+\frac{6}{8\cdot 3})\theta_1$
\item $\lambda^*= [(10,9,8,7,6,5,4,3,2,1)(10,9,8,5,4,1), (2),\emptyset,\emptyset]$\\
$e: \frac{13}{48},\frac{12}{43},  \frac{13}{38},\frac{12}{33},\frac{12}{28}, \frac{10}{23}, \frac{8}{18}, \frac{6}{13}, \frac{5}{8}, \frac{3}{3}$\\
$ ( \frac{14}{53}+\frac{17}{53\cdot 48})\theta_{10}, \frac{17}{48\cdot 43}\theta_9,(\frac{2}{43}+\frac{27}{43\cdot 38})\theta_8,\frac{27}{38\cdot 33}\theta_7,(\frac{2}{33}+\frac{4}{33\cdot 28})\theta_6,  +\frac{4}{28\cdot 23}\theta_5,\frac{4}{23\cdot 18}\theta_4,   \frac{4}{13 \cdot 8}\theta_3, \\
( \frac{1}{13}+\frac{9}{13 \cdot 8})\theta_2,
 (\frac{1}{8}+\frac{6}{8\cdot 3})\theta_1$
\item $ \bar \lambda= [(10,9,8,7,6,5,4,3,2,1)(10,9,8,5,4,1), (2,1),\emptyset,\emptyset]$\\
$e: \frac{13}{48},\frac{12}{43},  \frac{13}{38},\frac{12}{33},\frac{12}{28}, \frac{10}{23}, \frac{8}{18}, \frac{6}{13}, \frac{6}{8}, \frac{3}{3}$\\
$ ( \frac{14}{53}+\frac{17}{53\cdot 48})\theta_{10}, \frac{17}{48\cdot 43}\theta_9,(\frac{2}{43}+\frac{27}{43\cdot 38})\theta_8,\frac{27}{38\cdot 33}\theta_7,(\frac{2}{33}+\frac{4}{33\cdot 28})\theta_6,  +\frac{4}{28\cdot 23}\theta_5,\frac{4}{23\cdot 18}\theta_4,   \frac{4}{13 \cdot 8}\theta_3, \\
( \frac{3}{13}+\frac{6}{13 \cdot 8})\theta_2,
 \frac{6}{8\cdot 3}\theta_1$
\item $ \tilde \lambda= [(10,9,8,7,6,5,4,3,2,1)(10,9,8,5,4,1), (2,1),(1),\emptyset]$\\
$e: \frac{13}{48},\frac{12}{43},  \frac{13}{38},\frac{12}{33},\frac{12}{28}, \frac{10}{23}, \frac{8}{18}, \frac{6}{13}, \frac{7}{8}, \frac{3}{3}$\\
$ ( \frac{14}{53}+\frac{17}{53\cdot 48})\theta_{10}, \frac{17}{48\cdot 43}\theta_9,(\frac{2}{43}+\frac{27}{43\cdot 38})\theta_8,\frac{27}{38\cdot 33}\theta_7,(\frac{2}{33}+\frac{4}{33\cdot 28})\theta_6,  +\frac{4}{28\cdot 23}\theta_5,\frac{4}{23\cdot 18}\theta_4,   \frac{4}{13 \cdot 8}\theta_3,\\ 
( \frac{5}{13}+\frac{3}{13 \cdot 8})\theta_2,
 \frac{3}{8\cdot 3}\theta_1$
\end{itemize}
Note the changes in the oscillators caused by transit paths, from $4$ to $17$, from $4$ to $27$ and, in $\lambda^*$, from $6$ to $9$.
\end{examp}

Our aim is to prove that  every non-increasing strongly residue-homogeneous multipartition $\lambda$ has a unidirectional  Littelmann path.  This Littelmann path will have a standard LS-representation corresponding to $\lambda$. The proof will proceed by induction. 
The assumption will imply that the Littlemann path consists of long paths separated by oscillating paths.   Given a strongly residue-homogeneous multipartition $\lambda$ whose addable nodes are all of residue $\epsilon$, we recall that we defined $\lambda^*$ to be the multipartition obtained by adding $\epsilon$-addable nodes down to the last non-zero row, and if the last row of $\lambda$ is odd, we defined $\bar \lambda$  to be the multipartition obtained by adding nodes of residue $\epsilon$ in place of all the $\epsilon$-addable nodes down to the end of the last non-empty partition, including a singleton at the end. By applying the Littelmann algorithm for constructing $f_\epsilon$, we will want to show that $\lambda^*$, and if possible  $\bar \lambda$, both have standard Littelmann paths, and that if we continue to fill up empty partitions with singletons of residue $\epsilon$, the new multipartition $\tilde \lambda$ also has standard Littelmann path.

\begin{thm} In rank $e=2$, for highest weight $a\Lambda_0+b\Lambda_1$, and $a>0$, any $2$-regular non-increasing strongly residue-homogeneous multipartition  $\lambda$    has a Littelmann path $\pi_\lambda$  with  a standard LS-representation for $\lambda$ which satisfies in addition that
\begin{enumerate}
\item If some partitions are empty, there is a primary  tail whose length in the $0$-direction is the number of $0$-corner $\emptyset$, or, if every $0$-corner partition is non-empty, there is a secondary  tail whose length in the $1$-direction is the number of $1$-corner $\emptyset$ in $\lambda$. 
\item The Littelmann path is unidirectional.
\end{enumerate}
\end{thm}

\begin{proof} We are trying to prove this theorem only in the case of $-$, since, as we will show in the last section, it is not generally true in the case of $+$, so we will simplify the notation by letting $d_m=d^-_m$ and $\theta_m=\theta^-_m$.

We will show by induction on the length of the top row $n$ that the Littelmann path is standard for its multipartition. Using Lemmas 4.3 and \ref{q}, this will give us a unidirectional Littelmann path, where we will take the $m_i$ to be the numbers $n_i'$ for $i \leq k$,
the integral length of the long path with index $i$ will be $I_{n_i'}$, the oscillators will be the $C_i$ from Lemma \ref{q}. If the last segment, which has index  $k$, ends in a singleton there will be no transit path and $m_{k+1}$ will be the index of the floor, which will be $0$ or $1$ if there are empty partitions, but may be larger if there are no empty partitions. If the last segment does not end in a singleton, there will be a transit path. 

 Relying on Lemma \ref{built}, we assume that we have a $\lambda$ with standard LS-representation which is unidirectional, and using the operations given there we will do either a widening to get $\lambda^*$, a deepening appending a singleton  to get $\bar \lambda$,  or a deepening which adds solitary partitions (that is, with a single part of length $1$) to get $\tilde \lambda$. Furthermore, our multipartitions can have several segments, so we must do an induction on the segments as well.  Our hypothesis on this inner induction is that all the segments previously treated consist of long paths with oscillators,  with coefficients as in a standard LS-representation.  

 For all segments $i$ with $i<k$, the multipartitions and the Littelmann path for all three operations will be identical, so to condense the notation, we will take the multipartition $\lambda$ and add the $ \epsilon$-addable nodes of that segment, of which there are  $I_{n_i'}$, to get the new segment in $\lambda^*$.  We then perform the operation $f_\epsilon$ to the corresponding long path of $\pi_{\lambda}$, and check that the various invariants  $e_m, m_i,$ and $C_i$ of the segment $i$  match up to the new invariants  $e_m^*, m_i^*$ and $C_i^*$. Then for $i=k$, we must take  $e_m, I_m,  m_k,$ and $C_k$ and use them to find the new  $e_m^*, I_m^*, m_k^*,$ and  $C_k^*$ in the case of widening, the new $\bar e_m, \bar I_m,\bar m_k,$ and $ \bar C_k$ in the case of deepening, as well as the new $\tilde e_m, \tilde I_m, \tilde m_k $ and $\tilde C_k'$ in the case of a deepening which adds extra solitary partitions.

In all cases, for $1 \leq i \leq k$, the addition of all the $\epsilon$- addable  nodes will add nodes to each row and  one node to each of the segment boundaries, so that we will have 
$(n_i')^*=n_i'+1$. In a segment which is entirely internal to a partition and does not end in a singleton,  the addable nodes will be added to existing rows, and we will have 
$I_{n_i'}=I_{(n_i')^*}^*$.  However, if the segment is not entirely internal to a partition, the number of addable nodes not added to previous rows is $v_i+z_i$ from the bottom of the segment to the top of the multipartition, and similarly from the bottom of segment $i-1$ to the top of the multipartion we get  $v_{i-1}+z_{i-1}$ addable nodes at the boundaries between partitions. Thus in fact we must add the segment boundary adjustment $s_i=(v_i+z_{i})-(v_{i-1}+z_{i-1})$ defined just before Corollary \ref{segnot}.
\[
I_{(n_i')^*}^*=I_{n_i'}+s_i
\]
We divide into  four cases: for all $i$, continuing to widen or deepen the segment $i$, and for $i=k$, to widen after deepening, to deepen after widening and finally to adding  new solitary partitions.

\begin{itemize}

\item $1 \leq  i  \leq k$, where if the segment ends in a singleton, we deepen and if does not end in a singleton, we widen.   For $i<k$, there is no difference between widening and deepening, and what we do is add all the addable nodes in  each segment.  Thus it will suffice to calculate  the new  $e_m^*, m_i^*, I_{m_i}^*$ and  $C_i^*$, with some special treatment for the case $i=k$.  Each segment  has its own long path, of which the length is the number of addable rows in the segment. If we are at $m=n_i'$, then we reflect the entire long path.  As in Lemma \ref{q}, this produces a new long path in the new direction, that of $\theta_{m+1}$. By the induction hypothesis, $\lambda$ has a standard LS representation, which depends on the sequence $e_m=\frac{c_m}{d_m}$, $1 \leq m \leq n $. In order to prove that the widened multipartition $\lambda^*$ has a standard LS-representation, we have to show that the new Littelmann path we get by the various reflections and translations given by Littelmann has an LS-representation given by a sequence  $e_m^*=\frac{c_m^*}{d_m}$, for $1 \leq m \leq n+ 1$. Now in fact, $c_m$ is only changed where nodes are added in the $m$-ladder.  For each segment $i$, the nodes are added in one particular ladder, that corresponding to $n_i'+1$. Thus for all remaining $m$ in segment $i$, we have $c_m=c_m^*$, which means that for all $q_m$ which do not involve $n_i'+1$, $q_m^*=q_m$. Furthermore, since we already showed that $I_{n_i'}$ is the number of addable nodes in segment $i$, we have, for $m=n_i'$, that  
\[
 c_{m+1}^*=c_{m+1}+I_{m}
\]

We now calculate the new $q_m^*$ and  $q_{m+1}^*$.
\begin{align*}
 q^*_m=&e_m^*-e_{m+1}^*\\=&\frac{c_m^*}{d_m}-\frac{c_{m+1}^*}{d_{m+1}}\\
=&\frac{c_m}{d_m}-\frac{c_{m+1}+I_m}{d_{m+1}}\\
=&q_m-\frac{I_m}{d_{m+1}}\\
=&\left (\frac{I_m}{d_{m+1}}+\frac{C_i}{d_md_{m+1}}\right )-\frac{I_m}{d_{m+1}}\\
=&\frac{C_i}{d_md_{m+1}}
\end{align*}
This is the same oscillator that occurs in $q_{m-1}$, so we still have an oscillating part,  the new oscillator $C_i^*$ is identical to $C_i$ and $I_m^*$ becomes $0$ as we expect in the oscillating part.

We now calculate the coefficicient of $\theta_{m+1}$ in the new Littelmann path obtained by reflecting the long path.  The  part which is reflected is the subpath $\frac{I_m}{d_{m+1}}\theta_m$, so after $I_m$ reflections it becomes $\frac{I_m}{d_{m+1}}\theta_{m+1}$.  It joins the small oscillating part from the previous segment,  $\frac{C_{i-1}}{d_{m+1}d_{m+2}}\theta_{m+1}$ so that the coefficient of $\theta_{m+1}$ will be $\frac{I_m}{d_{m+1}}+\frac{C_{i-1}}{d_{m+1}d_{m+2}}$.

\begin{align*}
 q^*_{m+1}=&e_{m+1}^*-e_{m+2}^*\\=&\frac{c_{m+1}^*}{d_{m+1}}-\frac{c_{m+2}^*}{d_{m+2}}\\
=&\frac{c_{m+1}+I_m}{d_{m+1}}-\frac{c_{m+2}}{d_{m+2}}\\
=&\frac{I_m}{d_{m+1}}+q_{m+1}\\
=&\frac{I_m}{d_{m+1}}+\frac{C_{i-1}}{d_{m+1}d_{m+2}}
\end{align*}

 The number of addable nodes in the new segment should be the number of addable nodes in the segment $i$ of $\lambda$ plus the number of $s_i$ of partition boundaries in the  segment, which is $s_i=v_i-v_{i-1}+z_i-z_{i-1}$. so that 
\[
I_{m+1}^*=I_m+s_i
\]
We now rewrite $q_{m+1}^*$ in terms of the invariants of $\lambda^*$, where at the end we apply Corollary \ref{dif} and then the result we just proved that $C_i=C_i^*$.

\begin{align*}
 q^*_{m+1}=&\frac{I_m}{d_{m+1}}+\frac{C_{i-1}}{d_{m+1}d_{m+2}}\\
=&\frac{I_m}{d_{m+2}}+\frac{C_{i-1}+(d_{m+2}-d_{m+1})I_m}{d_{m+1}d_{m+2}}\\
=&\frac{I_m+s_i}{d_{m+2}}+\frac{C_{i-1}+(d_{m+2}-d_{m+1})I_m-d_{m+1}s_i}{d_{m+1}d_{m+2}}\\
=&\frac{I_{m+1}^*}{d_{m+2}}+\frac{C_{i-1}+rI_m-d_{m+1}s_i}{d_{m+1}d_{m+2}}\\
=&\frac{I_{m+1}^*}{d_{m+2}}+\frac{C_i}{d_{m+1}d_{m+2}}\\
=&\frac{I_{m+1}^*}{d_{m+2}}+\frac{C_i^*}{d_{m+1}d_{m+2}}\\
\end{align*}

 The little left-over piece at the end after we reflected most of the long path of segment $i$ will be $\frac{C_i}{d_md_{m+1}}\theta_m$.  Whereas the oscillating paths used to pair in a way which always returned them to the same integer value of the positive coordinate, they will now pair in a way which always returns them to the same integer value of the other coordinate.  If $i<k$, then the last piece will now join the next long path. For the case of deepening a segment  with $i=k$ ending in a singleton, the entire calculation is valid if we just replace  $  q_m^*$ with $ \bar q_m$ and similarly for $m+1$.

Note that if $i=1$, there was no previous small path, but in that case we also have $C_{i-1}=0$.
In summary, the long path of segment $i$ has shifted up one index from $m$ to $m+1$ as we passed from $\lambda$ to $\lambda^*$.

\item $i=k$, where the segment ends in a singleton and we want to widen.  The case of $i=k$ is considerably more complicated than the case of $i<k$ because our induction must take into account five separate sections of the Littelmann path, which we now list in order going away from the origin:  The long path of segment $k$, the oscillating part of segment $k$, the possibly empty transit path, the possibly empty terminal oscillator and the possibly empty tail.  The transit path can only be used as a seed to start a new segment when the last row is odd.

 We are assuming that $\lambda$  has a standard LS-representation and is unidirectional.
We now take as our inner induction hypothesis that $i=k$, and  that segment $k$ ends in a singleton. We wish  to do a widening and to compute the new oscillator.  Let $m=n_k'$.
We already showed in the previous section of this proof that if we were to do a deepening to get $\bar \lambda$, adding all $I_m$ addable nodes, we would get a new long path  with standard LS-representation, where
\[
\bar q_{m}=\frac{C_k}{d_md_{m+1}}
\]
\[
\bar q_{m+1}=\frac{I_{m}}{d_{m+1}}+\frac{C_k}{d_{m+1}d_{m+2}}
\] 
  Instead of adding all the $I_n$ addable nodes, we add only $I_n-z_k$  addable nodes.  In terms of the Littelmann path, this leaves a path of length  $1$ in the positive direction which will be the transit path and a continuation with the same terminal oscillator $C_k$ as before.
\[ \left ( \frac{z_k}{d_{m+1}}+\frac{C_k}{d_md_{m+1}} \right )\theta_m= \left ( \frac{C_k+z_kd_m}{d_md_{m+1}} \right )\theta_m
\]
Now we check that this widening will still give a standard LS-representation with  $C_k^*= {C_k+z_kd_m}$ as oscillator of segment $k$.  We have   $c_{m+1}^*=c_{m+1}+I_m-z_k$ and $c_\ell=c_\ell^*$ for  all other $c_\ell$ in the segment, so we need check only $q_\ell$ for $\ell=m, m+1$.

\begin{align*}
 q^*_m=&e_m^*-e_{m+1}^*\\=&\frac{c_m^*}{d_m}-\frac{c_{m+1}^*}{d_{m+1}}\\
=&\frac{c_m}{d_m}-\frac{c_{m+1}+I_m-z_k}{d_{m+1}}\\
=&q_m-\frac{I_m-z_k}{d_{m+1}}\\
=&\left (\frac{I_m}{d_{m+1}}+\frac{C_k}{d_md_{m+1}}\right )-\frac{I_m-z_k}{d_{m+1}}\\
=&\frac{C_k+z_kd_m}{d_md_{m+1}}
\end{align*}

This is precisely the formula for the oscillator $C_k^*$ in Lemma \ref{q}, where now $z_i^*=0$, and the path after widening is still unidirectional.  Thus we have a short path which is ambiguous:  it can either be thought of as the length $1$ transit path with a continuation which belongs to the terminal ocillator, or as the first step in a new oscillating part. With $m=n_k'$,
\begin{align*}
 q^*_{m+1}=&e_{m+1}^*-e_{m+2}^*\\=&\frac{c_{m+1}^*}{d_{m+1}}-\frac{c_{m+2}^*}{d_{m+2}}\\
=&\frac{c_{m+1}+I_m-z_k}{d_{m+1}}-\frac{c_{m+2}}{d_{m+2}}\\
=&\frac{I_m-z_k}{d_{m+1}}+q_{m+1}\\
=&\frac{I_m-z_k}{d_{m+1}}+\frac{C_{k-1}}{d_{m+1}d_{m+2}}\\
=&\frac{I_m-z_k}{d_{m+2}}+\frac{(d_{m+2}-d_{m+1})(I_m-z_k)+C_{k-1}}{d_{m+1}d_{m+2}}\\
=&\frac{I_m-z_k+s_k^*}{d_{m+2}}+\frac{r(I_m-z_k)-d_{m+1}s_k^*+C_{k-1}}{d_{m+1}d_{m+2}}\\
=&\frac{I_{m+1}^*}{d_{m+2}}+\frac{C_k^*}{d_{m+1}d_{m+2}}\\
\end{align*}
\noindent where the last step is a result of Lemma \ref{dif} applied to $\lambda^*$,
 \[
C_k^*=rI_{{n_k'}^*}^*-d_{{n_k'}^*+1}{s_k}^*+C_{k-1}^*,
\]
\noindent  using the fact that $C_{k-1}=C_{k-1}^*$  since $k-1<k$, that $ I_{m+1}^*=I_m+s_k-2z_k$, that $s_k^*=s_k-z_k$   and that $d_{m+2} = d_{m+1}+r$.
Further widening, as we saw in the first part of the proof, will move the long paths up by one, and will add new short paths with the same oscillator.  The position of the transit path will be unaffected, so we will have $m_{k+1}=m_{k+1}^*$and the section from the transit path to the tail will only be translated, so it will memain unchanged.

\item  $i=k$.  Deepening after widening without adding a new partition.   The partition ended in a singleton when we first started to widen at $m=m_{k+1}$ and the transit path has a hub which is a multiple of $\theta_{m_{k+1}}$. 

Thus, after widening to get $\lambda^*$, if the last row is odd  we can continue one more step and reflect the part of length $1$ in the transit path to get $\bar \lambda$.  This becomes the new long path of segment $k+1$. There will not be a transit path in the new seed part until after the next widening. The tail and the oscillating part before it remain as they were, so the new path is unidirectional.  It remains to show that the new path has a standard LS-representation.  The ladders $\bar c_m$ are identical to those of $c_m^*$, until we come to $m=n_{k+1}'$. The last row being odd, this is not the beginning of a segment for $\ell=m,m+1$, so we have $c_\ell^*=c_\ell$.  In $\bar \lambda$, we add a new node in the  $m+1$-ladder, so that   $\bar c_{m+1}= c_{m+1}+1$.  We have already checked the LS-representation for ladders outside of $m$ for $\lambda^*$, so now we check $\bar q_m$, where $I_m$, the number of addable nodes in the new segment $k+1$,   is $2$.

\begin{align*}
 \bar q_m=&\frac{c_m}{d_m}-\frac{c_{m+1}+1}{d_{m+1}}\\
=&q_m-\frac{1}{d_{m+1}}\\
=&\left (\frac{1}{d_{m+1}}+\frac{C_{k+1}}{d_md_{m+1}}\right )-\frac{1}{d_{m+1}}\\
=&\left (\frac{C_{k+1}}{d_md_{m+1}}\right )\\
\end{align*}
\begin{align*}
 \bar q_{m+1}=&\frac{c_{m+1}+1}{d_{m+1}}-\frac{c_{m+2}}{d_{m+2}}\\
=&q_{m+1}+\frac{1}{d_{m+1}}\\
=&\left (\frac{1}{d_{m+1}}+\frac{C_{k}}{d_{m+1}d_{m+2}}\right )\\
=&\left (\frac{d_{m+2}}{d_{m+1}d_{m+2}}+\frac{d_m+C_{k+1}}{d_{m+1}d_{m+2}}\right )\\
=&\left (\frac{2}{d_{m+2}}+\frac{C_{k+1}}{d_{m+1}d_{m+2}}\right )
\end{align*}

\noindent  If the last row is even, we must either widen again or proceed to the next possibility, adding a new partition.

\item $i=k$.  Deepening to add new partitions. The new partitions must come from the tail.  In order to get to the tail, we must add all addable nodes between the last long path and the tail. There will be exactly one such addable node if there is a transit path, and none otherwise. If there is no transit path, then the solitary partitions will open a new segment and the addable node will belong to segment $k$.  If there is a transit path, and the addable node belongs to the same ladder as the solitary partitions, then the addable node will belong to the same segment as the new solitary partition. Under all other conditions, we will have to add two segments.  In order to treat all cases at once, if there is a transit path, we will call the segment with the solitary partitions $k+2$, and, when we finish deepening, rename it as $k+1$ if necessary.

It is only possible to add $y$ solitary partitions is there are empty partitions, which is to say, if there is a tail,  which will be  as always at $\theta_m$ for 
$m= 1_{v_{i-1} \geq a}$.  Only $c_{m+1}$ will be changed by the addition of these addable nodes. We have already shown that $\bar \lambda$ has a standard LS-representation and is unidirectional,  so as in the earlier parts of the proof, it only necessary to check $q_m$ and $q_{m+1}$, and see that they correspond to the Littelmann path obtained by reflecting a piece of the tail of length $y$ in the positive coordinate. 
It is also true that $e_m=1$.  This is by definition if $m=0$, and we can only have $m=1$ if $c_1=d_1=a$. We have $\tilde c_{m+1}=c_{m+1}+y$, and for the remaining relevant indices $u$ we have $\tilde c_u=c_u$.

\begin{align*}
 \tilde q_m=&\tilde e_m-\tilde e_{m+1}\\=&\frac{\tilde c_m}{d_m}-\frac{\tilde c_{m+1}}{d_{m+1}}\\
=&1-\frac{c_{m+1}+y}{d_{m+1}}\\
=&\frac{d_{m+1}-c_{m+1}-y}{d_{m+1}}
\end{align*}

\noindent which gives the new tail, possibly zero, and coincides with the tail calculated in Lemma \ref {dif}.  Now, in order to find the new terminal oscillator, we calculate

\begin{align*}
\tilde  q_{m+1}=&\frac{c_{m+1}+y}{d_{m+1}}-\frac{c_{m+2}}{d_{m+1}d_{m+2}}\\
=&q_{m+1}+\frac{y}{d_{m+1}}\\
=&q_{m+1}+\frac{2y}{d_{m+2}}+\frac{y(d_{m+2}-2d_{m+1})}{d_{m+1}d_{m+2}}\\
=&q_{m+1}+\frac{2y}{d_{m+2}}+\frac{y(-d_{m})}{d_{m+1}d_{m+2}}\\
\end{align*}

\noindent where the added piece  $\frac{2y}{d_{m+2}}$ gives two addable nodes for each of the added solitary partitions. If $m=1$,  since $d_1=a$, 
then the quantity $\frac{y(-a)}{d_{m+1}d_{m+2}}$ is precisely the difference $C_{k+2}-C_{k+1}$ from Lemma \ref{dif}. In the case of $m=0$, 
then $d_0=-b$ and again the quantity $\frac{y(b)}{d_{m+1}d_{m+2}}$ is the difference $C_{k+2}-C_{k+1}$ from Lemma \ref{dif}.

\end{itemize}

\end{proof}

To illustrate the theorem, we give  case where the partition crosses the boundary between the $0$-corner and $1$-corner partitions.
\begin{examp} Consider $e=2$, $a=2$, $b=3$. We calculate the $q_m$ in this format for a sequence of multipartitions. The second has a transit path at $\theta_2$, with $d_2=7$ and $13=C_1=C_2+d_2$. The first four multipartitions have tails; the last does not, but it does have three long paths, and terminal oscillator $0$.
\begin{itemize}
\item $ [(2,1),(2,1),\emptyset,\emptyset,\emptyset$]\\
$e:  \frac{4}{7}, \frac{2}{2},$\\
$   (\frac{6}{12}+ \frac{6}{12 \cdot 7})\theta_2, \frac{6}{7\cdot 2}\theta_1$
\item $ [(3,2,1),(3,2),\emptyset,\emptyset,\emptyset$]\\
$e:  \frac{5}{12}, \frac{4}{7}, \frac{2}{2},$\\
$   (\frac{6}{17}+ \frac{13}{17\cdot 12})\theta_3, \frac{13}{12 \cdot 7}\theta_2, \frac{6}{7\cdot 2}\theta_1$
\item $ [(4,3,2,1),(4,3),(1),(1),\emptyset$]\\
$e:  \frac{6}{17}, \frac{5}{12}, \frac{6}{7}, \frac{2}{2},$\\
$  (\frac{7}{22}+ \frac{13}{22\cdot 17})\theta_4, \frac{13}{17\cdot 12}\theta_3, (\frac{5}{12}+ \frac{2}{12 \cdot 7})\theta_2, \frac{2}{7\cdot 2}\theta_1$
\item $ [(5,4,3,2,1)(5,4,1), (2,1),(2,1),\emptyset$]\\
$e: \frac{7}{22}, \frac{6}{17}, \frac{10}{12}, \frac{6}{7}, \frac{2}{2},$\\
$  (\frac{8}{27}+\frac{13}{27\cdot 22})\theta_5, \frac{13}{22\cdot 17}\theta_4, (\frac{8}{17}+ \frac{2}{17\cdot 12})\theta_3, \frac{2}{12 \cdot 7}\theta_2, \frac{2}{7\cdot 2}\theta_1$
\item $ [(6,5,4,3,2,1)(6,5,2,1), (3,2,1),(3,2,1),(1)]$\\
$e: \frac{8}{27}, \frac{7}{22}, \frac{14}{17}, \frac{10}{12}, \frac{7}{7},$\\
$  (\frac{9}{32}+\frac{13}{32\cdot 27})\theta_6, \frac{13}{27\cdot 22}\theta_5,( \frac{11}{ 22}+ \frac{2}{22\cdot 17})\theta_4,  \frac{2}{17\cdot 12}\theta_3, (\frac{2}{12}+\frac{0}{12 \cdot 7})\theta_2 $.
\end{itemize}

\end{examp}

\section{Motivation and Open Problems}

To motivate the definition of the residue-homogeneous multipartitions, we use the block reduced crystal graph from \cite{AS}, further  developed in \cite{BFS}, in which the vertices are the weights 
$\Lambda-\alpha$ in $P(\Lambda)$, where two weights connected by an edge if we can find a basis element from the weight space of each such that they are connected by $f_i$ for some $i$. In Figure 1, we have drawn the case 
of $\Lambda=2\Lambda_0+\Lambda_1$, truncated at degree $17$. Vertices on the same horizontal line correspond to multipartitions of fixed degree $n$. The diagonal edges down to the left, which are called $0$-strings,  represent action by $f_0$ and the edges going down to the right, called $1$-strings,  represent $f_1$.  The defects can all be determined in this case from the fact that acting on a string  by a simple reflection preserves defect, the highest weight element has defect $0$, adding the null root $\delta$ adds the level $r=3$ to the defect, and $\defect(\Lambda-\alpha_0)=1$.  The number of multipartitions corresponding to a vertex is fixed for each defect.  In this case there is one multipartition for each vertex of defect $0$ or $1$, there are no vertices of defect $2$, and there are $2$ multipartitions for a vertex of defect $3$.  We have written in the multipartitions for the external vertices of defect $0,1$ and $3$. 
We originally began to study the unidirectional Littelmann path in the context of external vertices.

It follows from the work of Scopes \cite{Sc} and generalizations by Ariki and Koike \cite{AK}, Kleshchev and Brundan \cite{Kl}, that if we have a weight all of whose crystal elements are external, then acting on the block of the cyclotomic Hecke algebra  by the Weyl group in the direction increasing the degree will produce a Morita equivalence. For every defect, there is a degree after which all weights of that defect are external, so the external vertices are a sort of limit case.  The set of strongly residue homogeneous multipartitions contained almost all of the external vertices, as well as a number of other vertices which are definitely not external. We defined  strongly residue homogeneous multipartitions here only for the case $e=2$, though it is possible to make an extension to larger ranks. We also consider unidirectional Littelmann paths only for defect $0$ weights for which the reflection multiplying $\Lambda$ is $s_0$, the reflection through the $0$-th simple root $\alpha_0$. Since we are in the case $e=2$, the dominant integral weight $\Lambda$ is of the form $\Lambda=a\Lambda_0+b\Lambda_1$ and we set $r=a+b$, which in type A is the level.  
\noindent \begin{figure*}[h]
\centering
\includegraphics[scale=0.7]{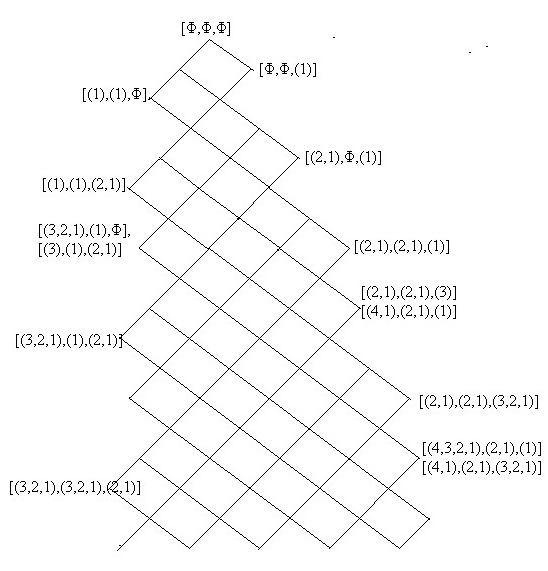}
\captionsetup{labelformat=empty}
\caption{Figure 1: $e=2,\Lambda=2\Lambda_0+\Lambda_1$}
\end{figure*}

\begin{lem}
A vertex of the crystal  for which all multipartitions are residue-homogeneous multipartitions is external.
\end{lem}

\begin{proof}
Let $\lambda$ be a residue-homogeneous multipartition.  Since each partition is alternating, the rows in a partition all end with the same residue, and by the conditions on the parity of the first rows, this residue is the same for each partition.  Let $1-\epsilon$ be the common residue at the end of all non-zero rows, so that the addable nodes at the ends are all of residue $\epsilon$.   Since there are no $\epsilon$ removable rows, we must have that $e_\epsilon$ is zero.
\end{proof}

The multipartitions labelling external vertices in Figure $1$ are all strongly residue-homogeneous. Among them, the multipartition $[(3),(1),(2,1)]$ is  strongly residue-homogeneous  because a singleton not at the boundary between $0$- and $1$-corner partitions makes no restriction on the previous partition. 

The main theorem is severely limited, in that it only applies to the case $e=2$ and even there, only to the non-increasing partitions.  We can in  fact get some results for $e=3$ and perhaps higher, but we must restrict ourselves to sequences of defect $0$ weights which are periodic, that is to say, generated by words which are periodic n the Weyl generators.  The case of $3$-periodic words is addressed in the original Ph.D. thesis, \cite{O}. We again get long paths and short oscillating sections between them.  

To move in the other direction, from a unidirectional Littelmann path to its multipartition is harder, since the   $I_m$ and the $C_i$ are nonlinear functions of the multipartition parameters $v_i,  t_i,$ and $ z_i$. However, the nonlinear function, as we showed in Lemma \ref{dif}, is a two part piecewise-linear function.  Furthermore, the same lemma allows us to determine the $s_i$, since, in the definition of unidirectional, we specify the $C_i$ and 
$I_{n_i'}$. In our two main examples, we used  two distinct primes as $a$ and $b$, which made it easy to identify the terminal oscillators, as small multiples of $a$ or $b$, while the widening oscillators were much larger.

In order to explain what interferes with getting a standard LS-representation in the ``+'' case, where the first $1$-corner partition has first row larger than the number of rows in the previous partition, as in $[(1),(1),(2,1)]$, we look an easy case

\begin{defn}
 A \textit{pseudo-floor} of $\lambda$ is a multipartition which begins like a defect $0$ multipartition with weight  $\psi^\pm _\ell$, but is  truncated  after $w$ partitions. 
\end{defn}

 If it is increasing at the bundary between $0$-corner and $1$-corner partition, it will not have a standard LS-representation.
Suppose we have a pseudo-floor $\lambda$ and do a deepening to $\bar \lambda$ We need to show that $\bar e_n \geq  \bar e_{n+1}$.  This is equivalent to showing the $q_{n} \geq 0$ . Since $n<n_1' =n+1$, we know from Lemma \ref{ri} for $\bar \lambda$ that $I_n=0$.  Thus we need to determine if $C_1\geq 0$.  We have $i=1=k$ since there is only one segment and $z_i=1$ because we have a pseudo-floor, so we are in case the second case in Lemma \ref{q}.
\[
C_1= \mp(v_1+1)b\pm r((v_1-a)_++1_{v_1 \geq a})
\]
\noindent In the $+$ case, $a<w<r$.  We then have $(v_1-a) \geq 0$, so we can drop the $+$.  Substituting
$v_1+1=w$ and using $b=(r-w)+(w-a)$, we have 
\[
C_1= -wb+ r(w-a)=(r-w)(w-a-w)=-(r-w)a<0.
\]

Thus only the ``-'' case will give us an LS-representation.  Since the formulae for a pseudo-floor are simple, we compute both the oscillator and the tail:

\[
C_1=
\begin{cases}
wb,&w\leq a,\\
(r-w)a,& r>w>a,
\end{cases}
\]

Also for a pseudo-floor, we can calculate the hub of the tail

\[
hub(T)=
\begin{cases}
\frac{a-w}{a}\theta_0,&w< a,\\
\frac{r-w}{b}\theta_1,& r>w \geq a
\end{cases}
\]

\begin{examp} Let $\Lambda=2\Lambda_0+\Lambda_1$ and let $\lambda$ be the multipartition $[(5,2,1),(1),\emptyset]$.

\begin{itemize}
\item $\lambda= [(5,2,1),(1),\emptyset$]\\
$e: \frac{1}{14},\frac{1}{11},\frac{3}{8},\frac{2}{5},\frac{2}{2}$\\
$   (\frac{1}{17}+ \frac{3}{17\cdot 14})\theta_5, \frac{3}{14 \cdot 11}\theta_4, (\frac{3}{11}+ \frac{1}{11\cdot 8})\theta_3, \frac{1}{8 \cdot 5}\theta_2, \frac{4+2}{5\cdot 2}\theta_1$
\item  $\lambda^*=[(6,3,2,1),(2),\emptyset]$\\
$e: \frac{1}{17},\frac{1}{14},\frac{4}{11},\frac{3}{8},\frac{3}{5},\frac{2}{2})\\$
$  (\frac{1}{20}+ \frac{3}{20\cdot 17})\theta_6, \frac{3}{17\cdot 14}\theta_5 ,(\frac{4}{14}+ \frac{1}{14 \cdot 11})\theta_4, \frac{1}{11\cdot 8}\theta_3, (\frac{1}{8}+ \frac{4}{8 \cdot 5})\theta_2, \frac{4}{5\cdot 2}\theta_1$
\item  $\bar \lambda=[(6,3,2,1),(2,1),\emptyset]$\\
$e: \frac{1}{17},\frac{1}{14},\frac{4}{11},\frac{3}{8},\frac{4}{5},\frac{2}{2},$\\
$  (\frac{1}{20}+ \frac{3}{20\cdot 17})\theta_6, \frac{3}{17\cdot 14}\theta_5 ,(\frac{4}{14}+ \frac{1}{14 \cdot 11})\theta_4, \frac{1}{11\cdot 8}\theta_3, (\frac{3}{8}+ \frac{2}{8 \cdot 5})\theta_2, \frac{2}{5\cdot 2}\theta_1$
\item $\tilde \lambda=[(6,3,2,1),(2,1),(1)]$\\
$e: \frac{1}{17},\frac{1}{14},\frac{4}{11},\frac{3}{8},\frac{5}{5},$\\
$  (\frac{1}{20}+ \frac{3}{20\cdot 17})\theta_6, \frac{3}{17\cdot 14}\theta_5 ,(\frac{4}{14}+ \frac{1}{14 \cdot 11})\theta_4, \frac{1}{11\cdot 8}\theta_3, (\frac{5}{8}+ \frac{0}{8 \cdot 5})\theta_2.$
\end{itemize}

In Figure $2$, we draw the hubs of the Littelmann paths in the case of deepening.

\noindent \begin{figure*}[h]\label{unidirectional}
\centering
\includegraphics[scale=0.5]{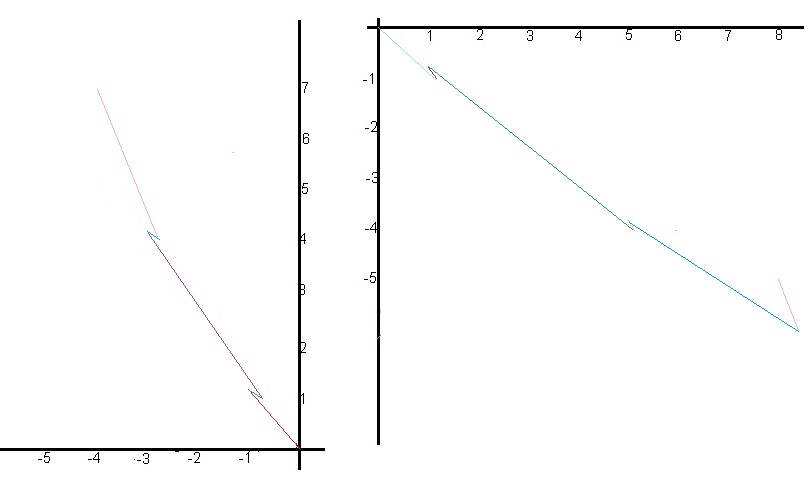}
\captionsetup{labelformat=empty}
\caption{Figure 2: $[(5,2,1),(1),\emptyset]$; $[(6,3,2,1),(2,1),\emptyset]$}
\end{figure*}

\noindent \smallskip{}

There are three long paths, but in the last path there is a piece of $1$-length $b=1$, which is actually the tail, and another piece of  $1$-length $2$ which will be needed if we wish to add two nodes to the partitions $(1)$.  If we choose to widen, then we get $\lambda^*=[(6,3,2,1),(2),\emptyset]$ in which only the first third of the last long path is reflected. If we choose to continue deepening to $\bar \lambda$, without opening another partition, then we will reflect the second third get a multipartition  $\tilde \lambda =[(6,3,2,1),(2,1),\emptyset]$ in which the tail is obvious.
Finally, if we continue to fill in the last non-zero partition, the tail will straighten out in the direction of the last long path, giving the Littelmann path of 
$\tilde \lambda$.
\end{examp}

\end{document}